\documentclass{amsart}
\usepackage{amsthm} 
\usepackage{amsmath}
\usepackage{amssymb}
\usepackage{comment}
\usepackage{graphicx}
\usepackage{enumitem}
\usepackage{tikz}
\usetikzlibrary{arrows}
\definecolor{red}{gray}{0.7}
\usepackage[section]{placeins}
\usepackage{caption} 

\newtheorem{theorem}{Theorem}[section]

\theoremstyle{definition}
\newtheorem{definition}[theorem]{Definition}

\theoremstyle{remark}
\newtheorem*{remark}{Remark}

\newcommand{\ZZ}{\mathbb{Z}}
\newcommand{\NN}{\mathbb{N}}
\newcommand{\TT}{\mathbb{T}}

\newcommand{\triqqq}{\Delta(4,4,4)}
\newcommand{\tridts}{\Delta(2,3,7)}

\DeclareMathOperator{\pr}{pr} 
\DeclareMathOperator{\dg}{deg} 

\begin{document}

\title[Cone types and spectral radii of hyperbolic tessellations]{Cone Types and Spectral Radius of Hyperbolic Triangle Groups and Hyperbolic Tessellations}

\author{Megan Howarth}
\author{Tatiana Nagnibeda}
\address{University of Geneva, Section of Mathematics, Rue du Conseil-Général 7-9, 1205 Geneva, Switzerland.}
\email{Megan.Howarth@unige.ch ; Tatiana.Smirnova-Nagnibeda@unige.ch}

\thanks{We acknowledge partial support by the Swiss NSF grant 200020-200400. The first author also acknowledges the support of NCCR SwissMAP}



\maketitle

\vspace{-1em}

\begin{abstract}
  This paper is devoted to the study of tessellations of the hyperbolic plane, especially the ones associated to hyperbolic triangle groups $\Delta(l,m,n)$. We give a full description of the cone types of these graphs and show that their number depends only on the defining parameters of the group. We then use the cone types structure to provide estimates of the spectral radius for the simple random walk on these tessellations, from above and from below.
\end{abstract}

\section{Introduction}

An important characteristic of a random walk on a graph of uniformly bounded degree is its spectral radius, which measures the rate of decay of probabilities of return to the starting point. By the celebrated theorem of Kesten and by its subsequent extensions, the graph is amenable, that is, fails the strong isoperimetric inequality, if and only if the spectral radius of the simple random walk on it is equal to $1$. In the non-amenable case, the exact value of the spectral radius is typically quite difficult to compute and explicit computations have only  been carried out for certain types of  trees and tree-like graphs. The spectral radius of the infinite $d$-regular tree  is equal to $2\sqrt{d-1}/d$. It is an intriguing problem to understand the nature of the spectral radius on infinite graphs which are not trees or quasi-trees.\\ 

Regular \textit{tessellations} of the plane appear in various areas of science since the nineteenth century and lead not only to beautiful pictures but also to deep mathematical results. A $\{k,d\}$-tessellation, with $k$ and $d$ positive integers, is, according to Schlegel, as cited by Coxeter \cite{Coxtessell}, an "infinite collection of regular $k$-gons, $d$ at each vertex, filling the whole plane just once". In order for the tessellation to be \textit{hyperbolic}, the parameters must satisfy the condition that
\begin{equation}
\label{eq:hypcond}
    (k-2)(d-2) > 4.
\end{equation}
Among the extensive literature on the subject, more recent papers relevant to our study include \cite{cannonpreprint, FloydPlotnick, BCS, christoforos, GrigorchukKravaris}. \\

In this paper, we concentrate our attention on hyperbolic tessellations which are Cayley graphs of hyperbolic triangle groups. By construction the latter tile the hyperbolic plane in such a way that three polygons meet per vertex and they may have one, two or three different numbers of sides, depending on the group.\\

The notion of \textit{cone types} was introduced by Cannon  in the $80$'s. He used it to compute growth functions of surface groups and some triangle groups \cite{cannonpreprint,CanWag} and ultimately to prove rationality of growth functions for all Gromov hyperbolic groups  \cite{cannon1984combinatorial}. Beside Gromov hyperbolic groups, Coxeter groups with Coxeter generating sets provide another class of Cayley graphs that have finitely many cone types, as was proven by Brink and Howlett \cite{brinkhowlett}.  A finite automaton describing  the structure of cone types is a  valuable combinatorial tool, enabling the study of infinite objects with only finite information. They have been used for example to compute the growth series of groups and graphs with finitely many cone types \cite{cannonpreprint, GriNag, BCS, GrigorchukKravaris}. Another interesting application is to the study of random walks and to the spectral theory \cite{Lyons, nag1, nag2, NagnibedaWoess, BCS, GiMu, Gouezel,  KLW, Pardo}.\\

The main objective of this paper is to describe combinatorially the cone types of hyperbolic triangle groups and subsequently make use of them to estimate the spectral radius of the simple random walk on these groups. This approach has been applied to the surface groups by the second author and by Gou\"ezel in \cite{nag2, Gouezel} to estimate the spectral radius respectively from above and from below. Different methods for estimating the spectral radius on surface groups from above were explored in \cite{BCCH}. Bartholdi and Ceccherini-Silberstein \cite{BCS} analysed cone types of $\{k,d\}$-tessellations with $k \geq 3$ and $d>3$ and suggested a method to estimate the spectral radius from below.  In slightly different but related directions, estimates for the drift of the simple random walk on these graphs were obtained by dynamical methods by Pollicott and Vytnova \cite{pollicott} and connected constants of self-avoiding random walks on them were estimated by Madras and Wu and by Panagiotis \cite{MWu, christoforos}. \\


The outline of the paper is the following. We start with Section \ref{sec:prelims} where we describe the framework and give the necessary background on graphs, cone types, hyperbolic tessellations, triangle groups and random walks. The main contribution of Section \ref{sec:conetypes} is Theorem \ref{thm:main}, describing the cone types of hyperbolic triangle groups. 
Section \ref{sec:specrad} utilises these cone types to give upper and lower bounds for the spectral radius of the simple random walks on the associated Cayley graphs. We illustrate our method with the groups $\Delta(4,4,4)$ and $\Delta(2,3,7)$. Finally, we summarise our numerical estimates in Section \ref{sec:conclusion} and compare them to the  combinatorial curvature of the corresponding graphs and to the estimates of the drift of the random walk from \cite{pollicott}.

\section{Preliminaries}
\label{sec:prelims}

\subsection{Cone types}
\label{subsec:graphs}

The graphs studied in this work are all assumed to be locally finite, infinite, connected and simple, although the definitions that follow are also applicable in a more general setting. To set the notation, let $\Gamma = (V,E)$ be such a graph, with vertex set $V$, edge set $E$ and fixed base point $x_{0} \in V$. \\

Let $x,y \in V$ and denote by $\lvert x \rvert$ the distance from $x_{0}$ to $x$. We say that $x$ is a \emph{predecessor} of $y$ if they are connected by an edge and $\lvert x \rvert < \lvert y \rvert$. In this case, $y$ is a \emph{successor} of $x$. \\
The \textit{degree} of $x \in V$, $\dg(x)$, is the number of edges emanating from $x$.

\begin{definition} 
\label{def:cone}
Let $x \in V$. The \emph{(graph) cone at the vertex $x$}, denoted $C(x)$, is the induced subgraph of $\Gamma$ rooted at $x$ such that its vertices form the set
\begin{equation*}
V(C(x)) = \{y \in V \mid x \textrm{ belongs to a geodesic joining } x_{0} \textrm{ to } y \}
\end{equation*}
and its edges are those of $\Gamma$ that join two vertices of $V(C(x))$. \\

\noindent Two vertices $x, y \in V$  have the \emph{same cone type} if their cones are isomorphic as rooted graphs. 
If the number of isomorphism classes of the cones $C(x), x \in V$, is finite, then $\Gamma$ is said to have \emph{finitely many cone types}.
 
\end{definition}


 \begin{remark}
 Cannon introduced the notion of cone type for Cayley graphs \cite{cannon1984combinatorial}; see also \cite{epstein1992word} where he defines the cone as a set of words in the generators and the isomorphism between the cones has to be realized by the group multiplication. These cone types are sometimes called \textit{word cone types} and in general they define a finer equivalence relation on the set of vertices than the graph cone types that we defined above and that we will use in this paper. 

 \end{remark}
 
It follows from Definition \ref{def:cone} that each vertex belongs to exactly one cone type and the latter determines both the number of its successors and their cone type. These cone types can be encoded in a matrix, defined as follows.

\begin{definition}
    The \textit{adjacency matrix} of the set of cone types of $\Gamma$ is a square matrix such that for all cone types $i, j$, one has
\begin{equation}
\label{eq:adjmat}
    M_{i,j} = \lvert \{y \in S(x) \mid y \text{ has cone type } j \} \rvert,
\end{equation}
where the vertex $x$ has cone type $i$ and $S(x) \subset V$ 
is the \textit{set of successors} of $x$.
\end{definition}

It follows that $M^{n}$ encodes the number of paths of finite length $n$ between each pair of vertices. If furthermore this matrix is Perron-Frobenius, i.e. if $M^{n}$ has only positive entries for some power $n$, then the set of cone types is said to be \textit{irreducible}. This property is equivalent to the cone of each vertex containing vertices of each type. It may happen that finitely many vertices in the vicinity of the base point $x_{0}$ do not satisfy this property; we discard them. The reduced set of cone types is denoted by $\TT$. \\



Another way to describe the adjacency of the cone types is to encode the information into a \textit{directed graph}, with the cone types as vertices and such that two vertices $i,j$ are joined by $M_{i,j}$ oriented edges from $i$ to $j$. Since two vertices of same cone type have same degree, denoted by $d_{i}$, the \textit{number of predecessors of a vertex of type $i$}, $r_{i}$, is well defined:
\begin{equation}
\label{eq:ri}
    r_{i} = d_{i} - \sum_{j} M_{i,j}.
\end{equation}
For each $i$ 
such that $r_{i} \neq 1$, we index the corresponding vertex by this number. For an example, see Figure \ref{fig:cta444}. Notice also that some vertices ($0$ and $1$ here) are not reachable from any other state; they correspond to the ones we delete to get the irreducible set of cone types.

\subsection{Tessellations of the hyperbolic plane}

Recall that a \emph{regular tessellation $\{k,d\}$} of the hyperbolic plane  partitions it into congruent $k$-gons, called \textit{tiles}, in such a way that exactly $d$ polygons meet at one vertex or along precisely one edge, with no gaps or overlapping, and satisfying \eqref{eq:hypcond}. \\


A family of examples of such tessellations is given by the triangular tessellations, where the tiles are triangles (so, $k=3$) whose angles are $\frac{\pi}{l}$, $\frac{\pi}{m}$ and $\frac{\pi}{n}$. In this case, we see by \eqref{eq:hypcond} that a triangular tessellation exists for each $d \geq 7$. 
Taking the planar dual tiling of this tessellation yields a tessellation of the hyperbolic plane by polygons, namely $2l$-, $2m$- and $2n$-gons. Its $1$-skeleton is the \textit{Cayley graph} of the corresponding triangle group $\Delta(l, m, n)$; see for instance Figure \ref{fig:444dual} where $\Delta(4,4,4)$ is illustrated. In fact, it has been shown by Magnus \cite{magnus1974noneuclidean} that such a triangle is indeed the fundamental domain of the group $\Delta(l,m,n)$ and that the symmetry group of this tessellation is generated by reflections in the sides of a triangle as above \cite{Coxtessell}. 

\begin{figure} 
    \centering
    \includegraphics[width=0.4\columnwidth]{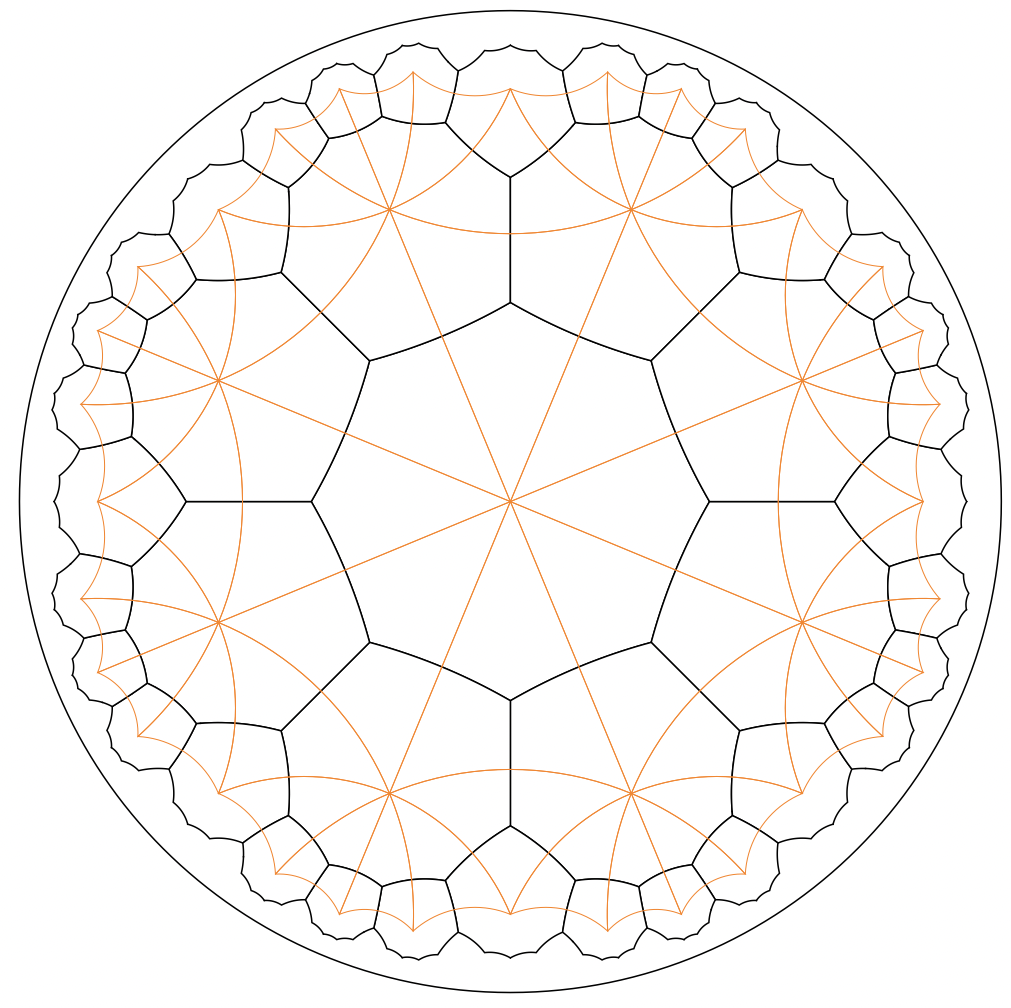}
    \caption{The two dual tessellations associated to $\Delta(4,4,4)$, \{3,8\} and \{8,3\}.}
    \label{fig:444dual}
\end{figure}


\begin{definition} 
Let $l, m, n \in \ZZ_{\geq 2}$ and denote by $\Delta$ a triangle defined by the angles $\frac{\pi}{l}, \frac{\pi}{m}$ and $\frac{\pi}{n}$. Embed $\Delta$ into the appropriate space according to the nature of the sum of its angles $\star = \frac{\pi}{l} + \frac{\pi}{m} + \frac{\pi}{n}$: the sphere (if $ \star > \pi$), the Euclidean plane (if $ \star = \pi$) or the hyperbolic plane (if $ \star < \pi$). Let $\overline{L}, \overline{M}, \overline{N}$ be the sides of $\Delta$, situated respectively opposite the angles $\frac{\pi}{l}, \frac{\pi}{m}, \frac{\pi}{n}$ and let $L, M, N$ be the reflections with respect to the sides $\overline{L}, \overline{M}, \overline{N}$, respectively. The \emph{triangle group $\Delta(l, m, n)$} is the group generated by the three reflections $L, M, N$.
\end{definition}



It can be shown that any triangle group $\Delta(l, m, n)$ admits the following presentation
\begin{equation}
\label{eq:prestri}
\Delta(l, m, n) = \langle L, M, N \mid L^{2} = M^{2} = N^{2} = (LM)^{n} = (MN)^{l} =(NL)^{m} = e \rangle
\end{equation} 
and moreover belongs to the class of Coxeter groups of finite rank, which we know to have finitely many cone types.


\subsection{Random Walks and Spectral Radius}

Let $\Gamma = (V,E)$ be a graph as in Subsection \ref{subsec:graphs}. Starting from a fixed base point $x_{0} \in V$, we successively take a step from a vertex $x \in V$ along one of the incident edges, towards one of its neighbouring vertices $y \in V$; this is the \emph{nearest neighbour random walk} (NNRW). In the particular case when each edge is chosen with equal probabilities $\frac{1}{\dg(x)}$, it is called the \emph{simple random walk} (SRW) on $\Gamma$. These probabilities can be encoded in a stochastic transition matrix 
\begin{equation*}
P = (p(x,y))_{x,y \in V}, \textrm{ where } p(x,y) \geq 0 \textrm{ and } \sum_{w \in V}p(x,w)=1 \textrm{ } \forall x \in V.
\end{equation*}
More generally we can also consider $p^{(n)}(x,y)$, the probability of reaching the vertex $y$ from the vertex $x$ \textit{in $n$ steps}, which is encoded in the $n$-th power of the transition matrix, $P^{(n)}$. The generating function of the transition probabilities is called the \emph{Green kernel}:
\begin{equation*}
\label{eq:green}
G(x,y \mid z) := \sum_{n = 0}^{\infty} p^{(n)}(x,y)z^{n},
\end{equation*}
where $x,y \in V$ and $z \in \mathbb{C}$. 
We are interested in determining its \emph{radius of convergence}, denoted by $R_{Gk}$, which is independent of $x,y$ when $\Gamma$ is connected and hence it is enough to study that of $G(x_{0}, x_{0} \mid z)$. We already know that $R_{Gk} \geq 1$ since for all $ \lvert z \rvert < 1$, 
\begin{equation*}
\label{eq:rgsmol1}
\sum_{n = 0}^{\infty} p^{(n)}(x_{0},x_{0})z^{n} \leq \sum_{n = 0}^{\infty} z^{n} < \infty, 
\end{equation*}
as $p^{(n)}(x_{0},x_{0}) \leq 1$. The inverse of $R_{Gk}$ is called the \textit{spectral radius} of the random walk 
\begin{equation}
\label{eq:spectralradius}
\rho = \frac{1}{R_{Gk}} = \underset{n \to \infty}{\textrm{lim sup}} \sqrt[n]{p^{(n)}(x_{0},x_{0})}.
\end{equation}
In $d$-regular graphs,  $\rho$ is indeed the spectral radius of $P$ viewed as a bounded self-adjoint operator on the Hilbert space $l^{2}(V)$.\\


Similarly to the Green kernel, we may consider the generating function of $q^{(n)}(x,y)$, the probability of reaching $y \in V$ from $x \in V$ \emph{for the first time after exactly $n$ steps}, given by
\begin{equation*}
\label{eq:bigf}
F(x,y \mid z) = \sum_{n=1}^{\infty} q^{(n)}(x,y)z^{n}.
\end{equation*} 
By definition, $F(x,x \mid 0) = 0$ for $x \in V$ and an important property of the function $F(x,x \mid z)$ is that it is analytic (thus continuous) in its disc of convergence. It is naturally related with the Green kernel within its disc of convergence by the formula
\begin{equation} \label{eq:grandf}
G(x,x \mid z) = \frac{1}{1-F(x,x \mid z)}.
\end{equation}
Since for every $x,y \in V$ and $n \geq 0$, $q^{(n)}(x,y) \leq p^{(n)}(x,y)$, the opposite inequality holds for the radii of convergence of their respective generating functions, $R_{F} \geq R_{Gk}$. More precisely, we deduce from above that 
they are related by \cite{nag2}
\begin{equation}
\label{eq:rg}
R_{Gk} = \begin{cases} R_{F} & \textrm{ if } F(x,x \mid R_{F}) \leq 1;\\
z_{0} & \textrm{ otherwise, }
\end{cases}
\end{equation}
where $z_{0} \in \mathbb{R}_{+}^{*}$ is the unique value for which $F(x,x \mid z_{0})=1$.

\section{Cone Types of Hyperbolic Triangle Groups}
\label{sec:conetypes}

We will exploit the symmetries of the hyperbolic tessellations to describe their cone types. Depending on the parameters $l,m$ and $n$, there are different cases to consider; so
for an accurate analysis, we partition the hyperbolic triangle groups into three families: when all three parameters are equal, $l=m=n$; when only two of them coincide; and when all three are different.

\begin{theorem}
\label{thm:main}
The number of cone types of a hyperbolic triangle group $\Delta(l,m,n)$ with standard presentation \eqref{eq:prestri} is as follows:
    \begin{enumerate}[label=(\roman*)]
        \item $\Delta(n,n,n)$ has $n+2$ cone types;
        \item[(ii.$1$)] $\Delta(l,n,n)$ has $l+2n+1$ cone types if $ l\geq 3 $;
        \item[(ii.$2$)] $\Delta(2,n,n)$ has $2n+5$ cone types;
        \item[(iii.$1$)] $\Delta(l,m,n)$ has $2(l+m+n)-2$ cone types if $ l\geq 3 $;
        \item[(iii.$2$)] $\Delta(2,m,n)$ has $2m+2n+7$ cone types if $m \geq 4$;
        \item[(iii.$3$)] $\Delta(2,3,n)$ has $2n+21$ cone types,
    \end{enumerate}
where $l$, $m$ and $n$ are distinct integers $\geq 2$, satisfying $\frac{1}{l} + \frac{1}{m} + \frac{1}{n} < 1$.
\end{theorem}

\begin{proof}

 We proceed\footnote{We could follow the geometric approach of \cite{ParkYau} as was done by the first author in \cite{meganmaster}, but it is better suited to describe the word cone types whereas for us it is enough to consider the graph cone types.} by analysing case by case. 

\begin{enumerate}[label=(\roman*)]
\item This first family with one defining parameter corresponds to the triangle groups of the form $\Delta(n,n,n)$, 
yielding a \textit{regular} $\{2n,3\}$-tessellation of the hyperbolic plane. Since the graph is bipartite, no two adjacent vertices have the same norm. Consequently, each $2n$-gon has a vertex of minimal norm (its \textit{base}) and one of maximal norm (its \textit{summit}). The summit being the only vertex that is not also the base vertex of any $2n$-gon, it is the only vertex with two predecessors and one successor. This yields one cone type for the identity, $n-1$ cone types placed symmetrically around the $2n$-gon, one cone type for the summit of the $2n$-gon and finally one cone type for its unique successor. Their corresponding adjacency matrix is
\begin{equation*}
\begin{pmatrix}
    0 & 3 & 0 & 0 & \dots & 0 & 0 & 0 \\
    0 & 0 & 2 & 0 & \dots & 0 & 0 & 0 \\
    0 & 0 & 1 & 1 & \dots & 0 & 0 & 0 \\
    0 & 0 & 1 & 0 & \dots & 0 & 0 & 0 \\
    \vdots &  &  &  & \ddots &  &  & \vdots \\
    0 & 0 & 1 & 0 & \dots & 0 & 1 & 0 \\
    0 & 0 & 0 & 0 & \dots & 0 & 0 & 1 \\
    0 & 0 & 0 & 2 & \dots & 0 & 0 & 0 \\   
\end{pmatrix}.
\end{equation*}


\vspace{1em}

\item The triangle groups $\Delta(l,n,n)$ defined by two parameters tile the hyperbolic plane with two sorts of polygons, $2l$-gons and $2n$-gons, with one $2l$-gon and two $2n$-gons meeting at each vertex. To accurately describe the cone types in this situation, we must distinguish two subcases: $l = 2$ and $l \geq 3$.\\

\textbf{Case 1: $\Delta(l,n,n), l \geq 3$.}
We study the graph inductively and start by looking at the three polygons (one $2l$-gon and two $2n$-gon) around the identity, which has its own unique cone type. It is clear that the cone types around the $2l$-gon, except that of its summit, should come in pairs, placed symmetrically around the polygon; this yields $l$ new cone types. The cone types around the $2n$-gons are however not symmetrical, since half of the edges separate a $2l$-gon and a $2n$-gon and the other half separate two $2n$-gons. We thus have $2n-2$ new cone types (as the identity and the one at the intersection of the $2l$-gon and $2n$-gon have already been counted). Searching further for new cone types to arise, we see that only vertices that are the successors of the summits of each of the two kinds of polygons add a new type each. Hence in total we get $1 + l + 2n$ cone types and we deduce the corresponding adjacency matrix in the same way as above. \\


\textbf{Case 2: $\Delta(2,n,n)$.} 
The parameter $l$ being equal to $2$ introduces two extra cone types that do not arise in Case (ii.$1$), which are due to the vertices that are both the summit of a $4$-gon and the successor of a summit of a $2n$-gon.\\


\item This third family corresponds to the tessellations of the hyperbolic plane associated to the groups $\Delta(l,m,n)$ defined by three different parameters, hence using three different polygons: $2l$-gons, $2m$-gons and $2n$-gons; we partition them into three subcases. \\ 

\textbf{Case 1: $\Delta(l,m,n), l \geq 3$.}
This case can be treated analogously to Case (ii.$1$) above, except that here we get $2l$ cone types (counting that of the identity) instead of $l$ coming from the $2l$-gon due to the loss of symmetry. Similarly, the $2n$- and $2m$-gons give rise to $2n-2$ and $2m-3$ cone types respectively, taking care not to count any twice. The only new cone types remaining are again the successors of the summits of the three polygons, thus yielding $2(l+m+n)-2$ cone types, which may naturally be described by their adjacency matrix. \\

\textbf{Case 2: $\Delta(2,m,n), m \geq 4$.}
We use the same method than for Case (ii.$2$) and find that the parameter $l=2$ yields the same particularities.\\


\textbf{Case 3: $\Delta(2,3,n)$.} This is the most technical case, due to the fact that $l=2$ and $\lvert l-m \rvert = 1$. It is treated analogously to the previous one but yields an even higher number of cone types, induced by the fact that many vertices are simultaneously the summit of a polygon and the successor of the summit of another.

\end{enumerate}

\end{proof}

\paragraph{Example} The Cayley graph of $\Delta(4,4,4)$ has $6$ cone types, which are described by the \textit{adjacency matrix} 
\begin{equation*}
    \begin{pmatrix}
        0 & 3 & 0 & 0 & 0 & 0 \\
        0 & 0 & 2 & 0 & 0 & 0 \\
        0 & 0 & \textbf{1} & \textbf{1} & \textbf{0} & \textbf{0} \\
        0 & 0 & \textbf{1} & \textbf{0} & \textbf{1} & \textbf{0} \\
        0 & 0 & \textbf{0} & \textbf{0} & \textbf{0} & \textbf{1} \\
        0 & 0 & \textbf{0} & \textbf{2} & \textbf{0} & \textbf{0}
    \end{pmatrix},
\end{equation*}
where the emboldened part corresponds to the reduced set of cone types, and by the \textit{associated digraph}, shown in Figure \ref{fig:cta444}. In fact, further work, as was done by the first author in \cite{meganmaster}, leads to the geometric representation of the cone types of such groups; see Figure \ref{fig:conetypes444}. 

\begin{figure}
    \centering
    \resizebox{8cm}{!}{
    \begin{tikzpicture}[->,>=stealth',auto,node distance=3cm,
  thick]

\node[red, style={circle, draw=red, thick, inner sep=4pt}] at (0,0)(0){0};
\node[red] at (0.4,-0.4)(index0){0};
\node[red, style={circle, draw=red, thick, inner sep=4pt}] at (2,0)(1){1};
\node[style={circle, draw=black, thick, inner sep=4pt}] at (4,0)(2){2};
\node[style={circle, draw=black, thick, inner sep=4pt}] at (6,0)(3){3};
\node[style={circle, draw=black, thick, inner sep=4pt}] at (8,0)(4){4};
\node at (8.4,-0.3)(index4){2};
\node[style={circle, draw=black, thick, inner sep=4pt}] at (10,0)(5){5};

\path[every node/.style={font=\sffamily\small}]

    (0) edge[red, bend left] node [right] {} (1)
    (0) edge[red] node [right] {} (1)
    (0) edge[red, bend right] node [right] {} (1)

    (1) edge[red, bend left] node [right] {} (2)
    (1) edge[red, bend right] node [right] {} (2)
    
    (2) edge[bend right] node [right] {} (3)
    (2) edge [loop above] node {} (2) 

    (3) edge[bend right] node [right] {} (2)
    (3) edge node [right] {} (4)
    
    (4) edge node [right] {} (5)

    (5) edge[bend right] node [right] {} (3)
    (5) edge[bend left] node [right] {} (3);

\end{tikzpicture}}
    \caption{The digraph associated to the set of cone types of the group $\Delta(4,4,4)$. The reduced set $\TT$ of cone types contains only the emboldened vertices.}
    \label{fig:cta444}
\end{figure}
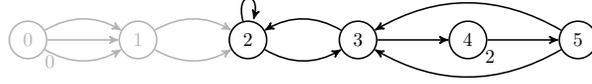

\begin{figure}
    \centering
    \resizebox{8cm}{!}{
    \includegraphics{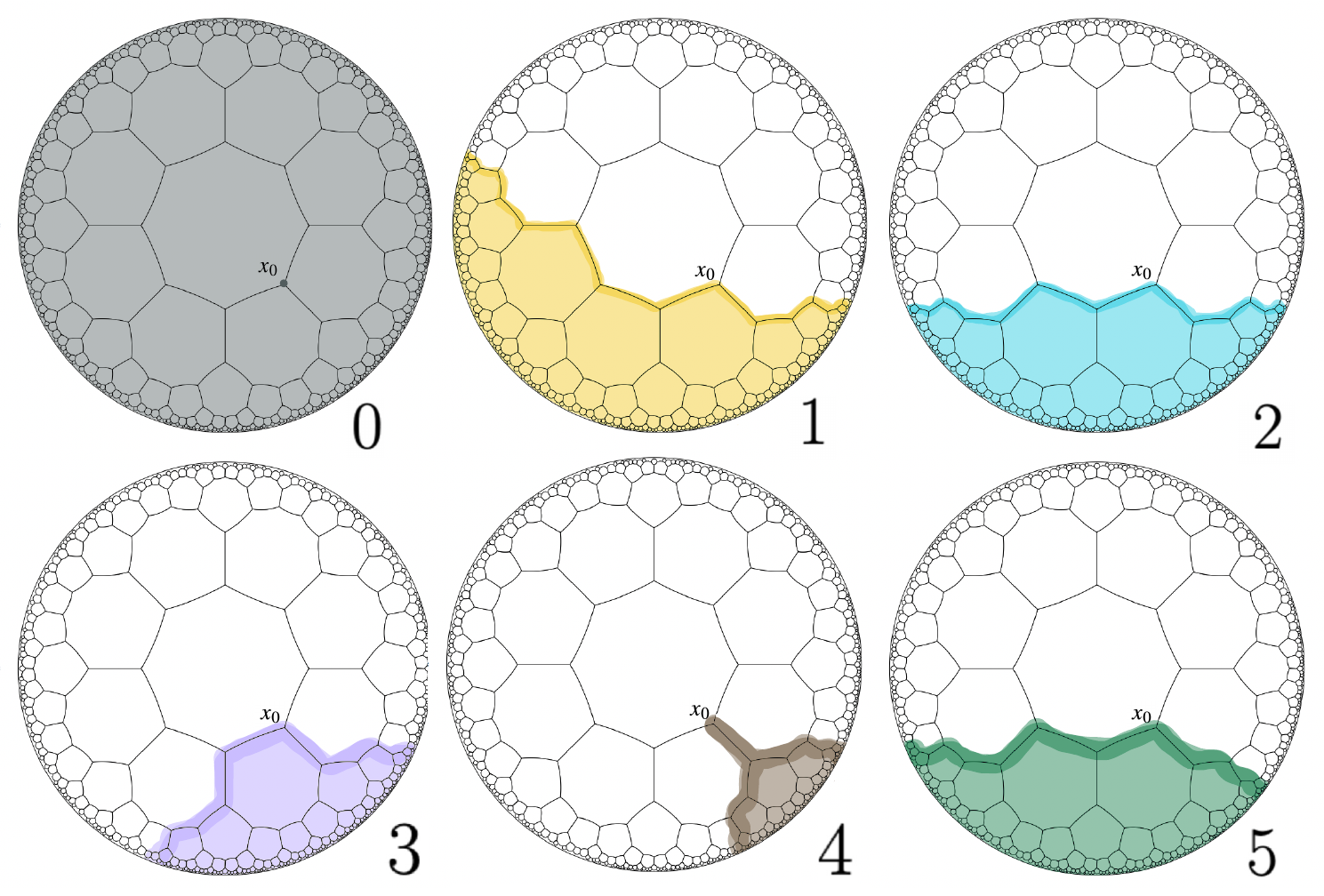}}
    \caption{The $6$ cone types of $\Delta(4,4,4)$.}
    \label{fig:conetypes444}
\end{figure}

\section{Estimating the Spectral Radius}
\label{sec:specrad}

The two methods we use to bound the spectral radius on a tessellation from above and from below are inspired by  the work on surface groups, respectively by the second author \cite{nag2} and by Gou\"ezel \cite{Gouezel}. We presently implement their methods to our considered tessellations of the hyperbolic plane associated to hyperbolic triangle groups.

\subsection{Upper Bound}

The method presented here was introduced in \cite{nag2} and can in principle apply in the context of locally finite graphs that have finitely many cone types. 
The idea is to reduce the study of the SRW on such a graph $\Gamma$ to that of a NNRW on the \textit{tree of geodesics}  associated to $\Gamma$, whose structure we will see is compatible with the cone types of $\Gamma$. The spectral radius in such trees is computable \cite{NagnibedaWoess}.


\begin{definition} \cite{nag2} The \emph{tree of geodesics} $T_{\Gamma}$ of $\Gamma$ has the set of vertices $V(T_{\Gamma})$  in one-to-one correspondence with the set of geodesic segments in $\Gamma$  of the shape $\{ [x_{0},x]_{x \in V} \}$, and two such vertices $\gamma_{1}, \gamma_{2} \in V(T_{\Gamma})$ are connected by an edge if one of the corresponding geodesic segments in $\Gamma$ is a one-step extension of the other. Its base vertex $\gamma_{0}$ corresponds to the empty geodesic segment $[x_{0},x_{0}]$ in $\Gamma$.
\end{definition}
By construction, there is a projection from $T_{\Gamma}$ onto $\Gamma$, given by
$\theta: [x_{0},x] \mapsto x$.
Since $\theta(C(\gamma)) = C(\theta(\gamma))$ for every $\gamma \in V(T_{\Gamma})$, one can derive a set of cone type for the vertices of $T_{\Gamma}$ from the one defined on the vertices on $\Gamma$, showing in particular that both graphs $\Gamma$ and $T_{\Gamma}$ have the same number of cone types. In what follows, suppose that $\Gamma$ has $K+1$ cone types with respect to the base point $x_0$. \\ 

Let us consider the NNRW on this tree defined as follows. Recall that in a tree rooted at $x_{0}$, each vertex $x \in V \setminus \{x_{0}\}$ has exactly one predecessor, denoted $\pr(x)$, and any two vertices are joined by a unique path. Furthermore, we impose compatibility with the cone types so that $p(x,\pr(x))$ depends only on the cone type of $x \in V$ and $p(x,y)$ depends solely on the cone types of $x \in V$ and any successor $y \in S(x)$. This way, we may denote the above probabilities respectively by $p_{-i}$ for a vertex $x$ of type $i \geq 1$ and by $p_{i,j}$ for vertices $x$ of type $i$ and $y \in S(x)$ of type $j$. Using these notations, the transition probabilities of the NNRW $P_{T}$ are
\begin{equation}
\label{eq:treeprobas}
\begin{cases} p_{i,j} = \frac{1}{d_{i}} \textrm{, for } i=0, \ldots, K; j= 1, \ldots, K;\\
p_{-i} = \frac{r_{i}}{d_{i}} \textrm{, for } i= 1, \ldots, K,
\end{cases}
\end{equation}
where $d_{i}$ and $r_{i}$, defined in \eqref{eq:ri}, are taken relatively to the original graph $\Gamma$. The following equalities then hold:
\begin{equation*}
\label{eq:probasct}
\begin{cases} 
p_{-i} + \sum\limits_{j=1}^{K} M_{i,j}p_{i,j} = 1 \textrm{, for } i=1,\ldots,K; \\
\sum\limits_{j=1}^{K} M_{0,j}p_{0,j} = 1,
\end{cases} 
\end{equation*}
where $M_{i,j}$ are the entries of the adjacency matrix of the cone types.\\

We now state the following theorem from \cite{nag2}, without proof.
\begin{theorem} 
Let $\Gamma$ be a graph as above, that is also bipartite. Denote by $\rho$ the spectral radius of the simple random walk on $\Gamma$ and by $\rho_{T}$ the spectral radius of the nearest neighbour random walk $P_{T}$ on its tree of geodesics. Then,
\begin{equation*}
    \rho \leq \rho_{T}.
\end{equation*}
\end{theorem}

Similarly to the transition probabilities $p(x, \pr(x))$ above, one can show that $q^{(n)}(x,\pr(x))$ depends only on the cone type of $x \in V \setminus \{ x_{0} \}$ for any $n \in \NN$, so we can unambiguously denote by $q^{(n)}_{-i}$ the probability of returning to $x_{0}$ for the first time after $n$ steps, starting from a vertex of type $i$. The corresponding generating function is then given by, for a vertex $x$ of type $i=0,\ldots,K$,
\begin{equation*}
F_{-i}(z) := F(x,\pr(x) \mid z) = \sum_{n=0}^{\infty} q^{(n)}_{-i} z^{n}.
\end{equation*}
These power series satisfy \cite{nag2} the following system of polynomial equations for $i=1,\ldots,K$:
\begin{equation} 
\label{eq:sys}
F_{-i}(z) = p_{-i}z + z \sum_{j=1}^{K} M_{i,j}p_{i,j}F_{-i}(z)F_{-j}(z).
\end{equation}
This means that for each $i=1, \ldots,K$, there exist polynomials $\mathcal{P}_{i}(w,z)$ such that 
\begin{equation}
\label{eq:polyF}
\mathcal{P}_{i}(F_{-i}(z),z) = 0.
\end{equation}
These functions $F_{-i}(z)$ are thus solutions of a polynomial functional equation and it follows that each one can be continued to an algebraic function for which its radius of convergence $R_{F_{-i}}$ is a singularity. So, for each index $i \in \{1, \ldots, K\}$, denote by 
 \begin{equation*}
    \mathcal{P}_{i}(w,z)= a_{0}^{(i)}(z)w^{n} + a_{1}^{(i)}(z)w^{n-1} + \ldots + a_{n}^{(i)}(z)
\end{equation*}
the corresponding polynomial for which \eqref{eq:polyF} holds. In practice it is obtained by performing the elimination of all variables bar one, say $w_{j}$, from the system of equations \eqref{eq:sys}. As the radius of convergence of the function $F_{-j}(z)$ is a (real and positive) singular point, it is known \cite{Ahlfors} that it is found either amongst the (real positive) roots of the polynomial $a_{0}^{(j)}(z)$, or amongst the (real positive) roots of the discriminant of $\mathcal{P}_{j}(F_{-j}(z),z)$. Moreover, when the set of cone types is irreducible, which we can always assume to be the case in the class of graphs that we consider, via a relabelling of the cone types,  the radii of convergence $R_{F_{-i}}$ all coincide for $i \in \{1, \ldots, K \}$. Finally, we deduce the radius of convergence of the Green kernel, first using the fact that the generating function  $F(x_{0}, x_{0} \mid z)$ can be expressed in terms of these functions, by
\begin{equation}
\label{eq:bigfminusf}
    F(x_{0}, x_{0} \mid z) = \sum_{j=1}^{K} M_{0,j}p_{0,j}zF_{-j}(z),
\end{equation}
and then applying the relation \eqref{eq:rg}. We conclude by taking its inverse to get the spectral radius as in \eqref{eq:spectralradius}.\\

\paragraph{Examples}
We illustrate this method with two hyperbolic triangle groups, $\Delta(4, 4, 4)$ and $\Delta(2, 3, 7)$. We choose these two examples as the first illustrates the simpler case of a regular tiling and the latter is the \textit{extreme} hyperbolic case, in the sense that it maximizes the sum $\frac{1}{l}+\frac{1}{m}+\frac{1}{n} < 1$. The first step is to use the description of the cone types from Section \ref{sec:conetypes} to construct their trees of geodesics, both depicted in Figure \ref{fig:geodtrees}. We then reduce the sets of cone types, before determining the transition probabilities for the NNRW on the tree of geodesics following equation \eqref{eq:treeprobas}. Given this data, we follow the steps explained above to compute their spectral radius.\\


 \begin{figure}[!tbp]
  \centering
  \begin{minipage}[b]{0.35\textwidth}
    \resizebox{15em}{18em}{%

\begin{tikzpicture}[every node/.style={circle, draw=black, thick, inner sep=4pt}]
\node at (0,0)(B){0};
\node at (-1,1)(A1){1};
\node at (1,1)(A2){1};
\node at (0,-1)(C1){1};
\node at (-1,-2)(D1){2};
\node at (1,-2)(D2){2};
\node at (2,-1)(C2){3};
\node at (-2,-3)(E1){3};
\node at (0,-3)(E2){2};
\node at (2,-3)(E3){2};
\node at (-3,-4)(F1){4};
\node at (-1,-4)(F2){2};
\node at (-3,-5)(G){5};
\node at (-4,-6)(H1){3};
\node at (-2,-6)(H2){3};

\draw[thick, black] (A1)--(B)--(A2);
\draw[thick, black] (B)--(C1)--(D1)--(E1)--(F1)--(G)--(H1);
\draw[thick, black] (D2)--(C1);
\draw[thick, black] (D2)--(C2);
\draw[thick, black] (D1)--(E2);
\draw[thick, black] (D2)--(E3);
\draw[thick, black] (G)--(H2);
\draw[thick, black] (E1)--(F2);
\end{tikzpicture}
}
     \end{minipage}
  \hfill
  \begin{minipage}[b]{0.55\textwidth}
    \resizebox{22em}{27em}{%
 
\begin{tikzpicture}[every node/.style={circle, draw=black, thick, inner sep=4pt}]

\node at (1,8)(Z1){21};
\node at (-1,8)(Z2){31};

\node at (4,7)(A1){20};
\node at (2,7)(A2){23};
\node at (0,7)(A3){30};

\node at (3,6)(B1){22};
\node at (1,6)(B2){11};
\node at (-1,6)(B3){20};

\node at (-5,5)(C1){19};
\node at (-3,5)(C2){6};
\node at (4,5)(C3){21};
\node at (2,5)(C4){21};

\node at (-4,4)(D1){26};
\node at (-2,4)(D2){4};
\node at (1,4)(D3){20};
\node at (3,4)(D4){24};

\node at (-3,3)(E1){25};
\node at (2,3)(E2){8};
\node at (4,3)(E3){25};
\node at (-1,3)(E4){4};

\node at (-2,2)(F1){9};
\node at (-1,2)(F2){27};
\node at (1,2)(F3){6};

\node at (-3,1)(G1){6};
\node at (0,1)(G2){1};
\node at (3,1)(G3){6};

\node at (-2,0)(H1){6};
\node at (0,0)(H2){0};
\node at (2,0)(H3){4};

\node at (-1,-1)(I1){2};
\node at (1,-1)(I2){3};
\node at (3,-1)(I3){20};

\node at (-2,-2)(J1){28};
\node at (1,-2)(J2){5};
\node at (-4,-2)(J3){7};

\node at (-3,-3)(K1){29};
\node at (-1,-3)(K2){20};
\node at (0,-3)(K3){7};
\node at (2,-3)(K4){6};

\node at (-4,-4)(L1){21};
\node at (-1,-4)(L2){10};
\node at (1,-4)(L3){4};
\node at (-9,-4)(L5){20};
\node at (-7,-4)(L6){14};
\node at (4,-4)(L7){25};

\node at (-2,-5)(M1){6};
\node at (0,-5)(M2){11};
\node at (-8,-5)(M3){34};
\node at (-6,-5)(M4){13};
\node at (3,-5)(M5){33};

\node at (0,-6)(N1){12};
\node at (-1,-6)(N2){13};
\node at (-3,-6)(N3){15};
\node at (-7,-6)(N4){19};
\node at (-5,-6)(N5){21};
\node at (2,-6)(N6){32};
\node at (4,-6)(N7){16};

\node at (-2,-7)(O1){14};
\node at (-4,-7)(O2){16};
\node at (-6,-7)(O3){18};
\node at (1,-7)(O4){31};
\node at (3,-7)(O5){15};

\node at (-5,-8)(P1){17};
\node at (-6,-9)(P2){20};

\draw[thick, black] (I1)--(H2)--(G2);
\draw[thick, black] (I1)--(H1);
\draw[thick, black] (H2)--(I2)--(H3)--(I3);
\draw[thick, black] (I1)--(J1);
\draw[thick, black] (I2)--(J2);
\draw[thick, black] (H3)--(G3);
\draw[thick, black] (J3)--(K1)--(J1)--(K2);
\draw[thick, black] (K3)--(J2)--(K4);
\draw[thick, black] (K1)--(L1);
\draw[thick, black] (L2)--(K3)--(L3);
\draw[thick, black] (M1)--(L2)--(M2)--(N1)--(N2)--(O1)--(N3)--(O2)--(P1)--(O3)--(N4)--(M3)--(L5);
\draw[thick, black] (P1)--(P2);
\draw[thick, black] (O3)--(N5);
\draw[thick, black] (N4)--(M4);
\draw[thick, black] (M3)--(L6);
\draw[thick, black] (N1)--(O4)--(N6)--(M5)--(N7);
\draw[thick, black] (N6)--(O5);
\draw[thick, black] (M5)--(L7);
\draw[thick, black] (F3)--(G2)--(F2)--(F1)--(E1)--(D1)--(C1);
\draw[thick, black] (F1)--(G1);
\draw[thick, black] (E1)--(D2);
\draw[thick, black] (D1)--(C2);
\draw[thick, black] (G2)--(F3)--(E2)--(D3)--(C4)--(B1)--(A2)--(B2)--(A3)--(B3);
\draw[thick, black] (E2)--(D4)--(E3);
\draw[thick, black] (D4)--(C3);
\draw[thick, black] (A2)--(Z1);
\draw[thick, black] (A3)--(Z2);
\draw[thick, black] (B1)--(A1);
\draw[thick, black] (F2)--(E4);

\end{tikzpicture}
}
  \end{minipage}
    \caption{Parts of the trees of geodesics of $\triqqq$ (left) and $\tridts$ (right).}
    \label{fig:geodtrees}
\end{figure}
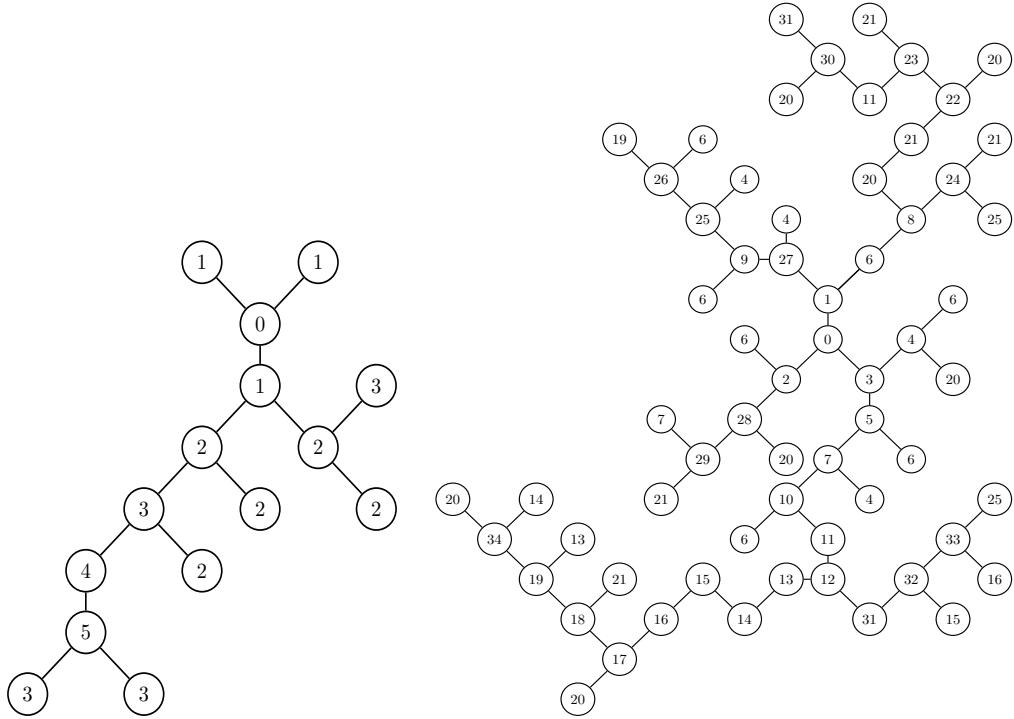

Let us give the detailed computations for $\triqqq$. As we consider the reduced set of cone types $\TT= \{2,3,4,5\}$, thus having in particular deleted type $0$, we must pick a new starting point $\tilde{\gamma_{0}}$ in such a way that its successors' cone types take value in the reduced set. We conveniently choose it to have cone type $4$, so the generating function  becomes
 \begin{equation}
 \label{eq:bigfqqq}
 F(\tilde{\gamma_{0}},\tilde{\gamma_{0}} \mid z) \overset{\eqref{eq:bigfminusf}}{=} 
 \frac{1}{3}zF_{-5}(z),
 \end{equation}
from which we deduce by \eqref{eq:grandf} the expression of the Green kernel in its disc of convergence. Then, the functions $F_{-i}(z)$ for $i \in \{2, 3, 4, 5\}$ satisfy the recursive relations of \eqref{eq:sys}, which together lead to the system of equations (using the notation $F_{-i}(z) = w_{i}$ for clarity)
\begin{equation*}
    \begin{cases}
    z + z w_{2}^{2} + z w_{2} w_{3} - 3w_{2} = 0\\
    z + z w_{2} w_{3} + z w_{3} w_{4} - 3w_{3} = 0\\
    2z + z w_{4} w_{5} - 3w_{4} = 0\\
    z + 2z w_{3} w_{5} - 3w_{5} = 0.\\
    \end{cases}
\end{equation*}
Using the program CoCoA5 \cite{CoCoA}, we perform the elimination of the variables $w_{2}, w_{3}$ and $w_{4}$ as only $w_{5}$ appears in \eqref{eq:bigfqqq}. It results that $w_{5}$ satisfies the polynomial equation
\begin{equation*}
\begin{split}
    \mathcal{P}_{5}(w_{5},z) & := 729 w_{5}^3 - 729 w_{5}^2 z - 486 w_{5}^4 z + 243 w_{5} z^2 - 324 w_{5}^3 z^2 + 
 \mathbf{81 w_{5}^5 z^2} \\ 
 & - 27 z^3 + 324 w_{5}^2 z^3  + 297 w_{5}^4 z^3 - 72 w_{5} z^4 + 117 w_{5}^3 z^4 \mathbf{- 36 w_{5}^5 z^4} \\
 & + 6 z^5 - 69 w_{5}^2 z^5 - 84 w_{5}^4 z^5 + 6 w_{5} z^6 + 16 w_{5}^3 z^6 + \mathbf{8 w_{5}^5 z^6} = 0,
\end{split}
\end{equation*}
of which we study the real positive roots of $a_{0}^{(5)}(z)$ (emboldened) and those of the discriminant of $\mathcal{P}_{5}(w_{5},z)$. There are several numerical candidates for $R_{F_{-5}}$, that all come from the roots of the discriminant of $\mathcal{P}_{5}(w_{5},z)$, so we use Mathematica to visualise locally the real parts of the solutions to $\mathcal{P}_{5}(w_{5},z) = 0$. From the picture we determine which branch corresponds to $F_{-5}(z)$, by comparing each one against the known properties of the $F_{-j}(z)$  functions. We conclude that its radius of convergence is equal to
\begin{equation*}
    R_{F_{-5}} = 1.0321531591.
\end{equation*}
As this is also the radius of convergence $R_{F}$ of $F(\tilde{\gamma_{0}}, \tilde{\gamma_{0}} \mid z)$, 
it remains to deduce that of the Green kernel. We verify that $F(\tilde{\gamma_{0}}, \tilde{\gamma_{0}} \mid R_{F}) < 1$, so by \eqref{eq:rg} we have $R_{F_{-5}} = R_{F} = R_{Gk}$ and thus finally, \eqref{eq:spectralradius} implies
\begin{equation*}
    \rho_{T_{\Delta(4,4,4)}} = \frac{1}{R_{Gk}} = 0.9688484613. 
\end{equation*}
We proceed in the same way for the other group $\Delta(2,3,7)$ and get the following result:
\begin{equation*}
    \rho_{T_{\Delta(2,3,7)}} = 0.9460344380.
\end{equation*}

\subsection{Lower Bound}

As every group on $d$ generators is a quotient of the free group with the same number of generators, there is a trivial lower bound to its spectral radius, 
    $\rho \geq \frac{2\sqrt{d - 1}}{d}.$
In the case of hyperbolic triangle groups, $d=3$, hence
    $\rho \geq 0.942809$.
To improve this lower bound estimate of the spectral radius, we use the method developed by Gou\"ezel in \cite{Gouezel} that can in principle apply to 
graphs with finitely many cone types. \\

Let $s_{n}(i)$ be the number of vertices of cone type $i$ on the $n$-sphere, then it is clear that it satisfies the following recurrence formula
\begin{equation*}
    s_{n+1}(i) = \sum_{j} \frac{1}{r_{i}} M_{j,i} s_{n}(j),
\end{equation*}
where $M_{j,i}$ are the entries of the adjacency matrix \eqref{eq:adjmat} of the set of cone types.
This suggests to consider the matrix $\tilde{M}$, defined by 
\begin{equation*}
    \tilde{M}_{i,j} = \frac{1}{r_{i}}M_{j,i},
\end{equation*}
so that 
    $s_{n+1} = \tilde{M} s_{n}$.
Define also the matrices 
\begin{equation*}
\label{eq:mprime}
    M'=D^{-\frac{1}{2}}M^{T}D^{\frac{1}{2}},
\end{equation*}
where $D$ is the diagonal matrix with entries $A_{i}$, a positive Perron-Frobenius eigenvector of $\tilde{M}$, and 
\begin{equation*}
\label{eq:Mprimeprime}
    M''=\frac{M'+M'^{T}}{2}.
\end{equation*}

We now state without proof the following theorem from \cite{Gouezel}.

\begin{theorem}
\label{thm:algou}
Let $\Gamma$ be a graph as above, that is also bipartite and regular of degree d. Let $\TT = \{1, \ldots, K\}$ be its irreducible set of cone types and $M$ the adjacency matrix of $\TT$.\\
Then, 
\begin{equation}
\label{eq:gou}
    \rho \geq \frac{2 \lambda}{d \sqrt{\nu}},
\end{equation}
where $\nu > 0$ and $\lambda > 0$ are the Perron-Frobenius eigenvalues of $\tilde{M}$ and $M^{\prime \prime}$, respectively.

\end{theorem}

\begin{remark}
    As shown in \cite{Gouezel}, this lower bound corresponds to the spectral radius of a random walk on the space $\ZZ \times \TT$, which is, using the notations introduced above, 
    defined by the transition probabilities
    \begin{equation*}
        p\left((n,i), (n+1,j)\right) =\frac{M_{i,j}}{d} \hspace{1em};\hspace{1em} p\left((n,i), (n-1,j)\right) = e^{-\nu}\frac{A_{j}M_{j,i}}{A_{i} d}.
    \end{equation*}
\end{remark}


\paragraph{Examples} We again illustrate this algorithm with the groups  $\Delta(4, 4, 4)$ and $\Delta(2, 3, 7)$. The description of the cone types given in Section \ref{sec:conetypes} gives all the data required to apply Theorem \ref{thm:algou}. \\

Taking the reduced set of cone types of $\triqqq$; the corresponding first matrices are 

\begin{equation*}
\label{eq:algoumatrixqqq}
M = 
\begin{pmatrix}
1 & 1 & 0 & 0\\
1 & 0 & 1 & 0\\
0 & 0 & 0 & 1\\
0 & 2 & 0 & 0\\
\end{pmatrix} 
\text{ and }
\hspace{0.3em}
\tilde{M} = 
\begin{pmatrix}
1 & 1 & 0 & 0\\
1 & 0 & 0 & 2\\
0 & 1/2 & 0 & 0\\
0 & 0 & 1 & 0\\
\end{pmatrix}.
\end{equation*}
We then derive from them the other three matrices that appear in the algorithm. Proceeding in the same way for the group $\tridts$, we obtain larger matrices, of size $24$. We use Python for all the computations of the maximal eigenvalues of the matrices $\tilde{M}$ and $M''$ and also to apply the final formula \eqref{eq:gou}. The numerical results are summarized in Table \ref{tab:spectralestimates}, in the next Section \ref{sec:conclusion}.

\section{Numerical Estimates for Some Examples}
\label{sec:conclusion}

There is a natural notion of \textit{combinatorial curvature} for a graph $\Gamma = (V,E)$ that embeds in the (hyperbolic) plane, in such a way that it is the $1$-skeleton of some (hyperbolic) tessellation. It is a function on the vertices, calculated for $v \in V$ via the formula \cite{BauesPeyerimhoff}
\begin{equation*}
    \kappa(v) = 2 \pi \left( 1 - \frac{\dg(v)}{2} + \sum_{f : v \in f} \frac{1}{E_{f}} \right),
\end{equation*}
where $E_{f}$ denotes the number of edges bounding the polygonal tile $f$. In the particular case when $\Gamma$ is the Cayley graph of a hyperbolic triangle group $\Delta(l,m,n)$, the value of this curvature is negative and constant on the vertices:
\begin{equation}
\label{eq:combcurvtri}
    \kappa(v) = -\pi \left( 1- \left( \frac{1}{l}+\frac{1}{m}+\frac{1}{n} \right) \right), \hspace{0.3em} \forall v \in V.
\end{equation}

We summarise in Table \ref{tab:spectralestimates} our upper and lower bounds for the spectral radius for $10$ examples of hyperbolic triangle groups. They are positioned in decreasing order of their combinatorial curvature \eqref{eq:combcurvtri}. We also include the estimates for the surface group of genus $2$ \cite{nag2, Gouezel}, as the comparison of the results is interesting. \\

A similar table can be found in \cite{pollicott}, where Pollicott and Vytnova compute upper and lower bounds for the \textit{drift} of the random walk on some hyperbolic triangle groups. We notice that the general tendency is for the drift to increase and the spectral radius to decrease, as the combinatorial curvature decreases.


\begin{table}
    \centering
\begin{tabular}{ c | c | c | c}
 \textbf{Group} & \textbf{Lower Bound} & \textbf{Upper Bound} & \textbf{Combinatorial curvature} \\ 
 \hline
$\Delta(2,3,7)$ & $0.997\textbf{4}952153$ & $0.997\textbf{9}155005$ & $-\pi/42$ \\
$\Delta(2,4,5)$ & $0.99\textbf{3}8397191$ & $0.99\textbf{4}7303685$ & $-\pi/20$ \\ 
$\Delta(3,3,4)$  & $0.98\textbf{8}1065017$ & $0.98\textbf{9}6253048$ & $-\pi/12$ \\
$\Delta(2,5,5)$ & $0.98\textbf{8}3635961$ & $0.98\textbf{9}2337907$ & $-\pi/10$ \\
$\Delta(2,6,6)$ & $0.98\textbf{2}5162138$ & $0.98\textbf{3}5349956$ & $-\pi/6$ \\
$\Delta(3,4,4)$  & $0.97\textbf{7}4836673$ & $0.97\textbf{8}9017112$ & $-\pi/6$ \\ 
$\Delta(3,4,5)$ & $0.97\textbf{2}4491846$ & $0.97\textbf{3}6926635$ & $-13\pi/60$ \\
$\Delta(4,4,4)$ & $0.96\textbf{7}6175845$ & $0.96\textbf{8}8484613$ & $-\pi/4$ \\ 
$\Delta(3,5,7)$ & $0.96\textbf{4}2297084$ & $0.96\textbf{5}1708503$ & $-34\pi/105$ \\
$\Delta(7,7,7)$ & $0.94\textbf{5}5418401$ & $0.94\textbf{6}0344380$ & $-4\pi/7$ \\
 
\hline
$G_{2}$ & $0.662\textbf{4}77$ & $0.662\textbf{8}153757$ & $-4\pi$ \\

\end{tabular}
   \caption{Numerical estimates of the spectral radius of some hyperbolic triangle groups and the surface group of genus $2$.}
    \label{tab:spectralestimates}
\end{table}

\section*{Acknowledgments}
The authors would like to thank Christoforos Panagiotis for interesting discussions around regular hyperbolic tessellations.

\bibliographystyle{abbrv}
\bibliography{refs}

\end{document}